\documentclass[reqno]{amsart}

\usepackage{amsmath,amsthm,amscd,amssymb,amsfonts, amsbsy}
\usepackage{latexsym, color, enumerate}
\usepackage{pxfonts}
\usepackage{mathrsfs}
\usepackage{hyperref}
\usepackage[margin=1.6in]{geometry}

\usepackage{tikz}

\numberwithin{equation}{section}

\theoremstyle{plain}
\newtheorem{theorem}{Theorem}[section]
\newtheorem{lemma}[theorem]{Lemma}
\newtheorem{corollary}[theorem]{Corollary}
\newtheorem{proposition}[theorem]{Proposition}

\theoremstyle{definition}

\newtheorem{assumption}[theorem]{Assumption}

\theoremstyle{remark}
\newtheorem{remark}[theorem]{Remark}

\newcommand{\bR}{\mathbb{R}}

\newcommand\It{\widetilde{I}}

\providecommand{\norm}[1]{\lVert#1\rVert}

\renewcommand{\vec}[1]{\boldsymbol{#1}}

\def\XXint#1#2#3{{\setbox0=\hbox{$#1{#2#3}{\int}$}
		\vcenter{\hbox{$#2#3$}}\kern-.5\wd0}}

\newcommand{\p}{\partial}
\newcommand{\epsi}{\varepsilon}
\newcommand{\tepsi}{\tilde{\varepsilon}}

\begin{document}
	
	\title[Quantitative Unique Continuation]{Quantitative unique continuation for Robin boundary value problems on $C^{1,1}$ domains} 

	\author[Z. Li]{Zongyuan Li}
	\address[Z. Li]{Department of Mathematics, Rutgers University, Hill Center - Busch Campus
		110 Frelinghuysen Road, Piscataway, NJ 08854, USA}	
	\email{zongyuan.li@rutgers.edu}

	\author[W. Wang]{Weinan Wang}	
	\address[W. Wang]{Department of Mathematics, University of Arizona, 617 N Santa Rita Ave, Tucson, AZ 85721, USA}
	\email{weinanwang@math.arizona.edu}


	\begin{abstract}
			In this paper, we prove two unique continuation results for second order elliptic equations with Robin boundary conditions on $C^{1,1}$ domains. The first one is a sharp vanishing order estimate of Robin problems with Lipschitz coefficients and differentiable, sign-changing potentials. This generalizes the result for the ``Robin eigenfunctions'' in \cite{ZHU2018}, which deals with the case with constant potentials. The second result is a unique continuation result from the boundary -- any non-trivial solution cannot vanish at infinite order from the boundary or vanish on an open subset on the boundary. Such result generalizes the one in \cite{AE} for the Laplace equation on $C^{1,1}$ domains with zero Neumann boundary conditions.
	\end{abstract}
	
	\maketitle
	
	\section{Introduction}
	In this paper, we prove some unique continuation properties of the second-order divergence-form elliptic equation
	\begin{equation}\label{eqn-211003-0419}
		\operatorname{div}(ADu) =  Vu  \,\,\text{in }\, \Omega \subset\bR^d,\quad d\geq 3,
	\end{equation}
	with Robin boundary condition
	\begin{equation}\label{eqn-21003-0459}
		ADu\cdot \vec{n} = \eta u \,\,\text{on }\, \p\Omega.
	\end{equation}	
	We first study the strong unique continuation property (SUCP), which states that any nontrivial solution to an elliptic equation cannot vanish at an infinite order. Such property generalizes the commonly known fact for analytic (holomorphic) functions on the complex plane, and has been extensively studied. For instance, in \cite{KT1}, a nearly optimal scaling invariant SUCP was proved for
	\begin{equation*}
		\operatorname{div}(ADu + Bu) + W\cdot Du + Vu =0\quad \text{in}\,\,\Omega \subset\bR^d
\end{equation*}
with $A$ being Lipschitz, $V\in L^{d/2}$, and $B, W$ in spaces close to $L^d$. More precisely, if for some point $x_0\in\Omega$,
\begin{equation*}
	|u(x)| = O(|x-x_0|^N)\quad\text{for any}\,\,N>0,
\end{equation*}
then $u\equiv 0$.
For the history of SUCP, one may refer to \cite{KT1} and the references therein.

In the current paper, we prove the quantitative SUCP. We aim to find the sharp upper bound for the vanishing order $N$ of nontrivial solutions to \eqref{eqn-211003-0419}. Let us start from Laplace eigenfunctions, i.e., solutions to
	\begin{equation*}
		\Delta u = -\lambda u.
	\end{equation*}
From the spherical harmonics, one could simply see that solutions can vanish as fast as $|x-x_0|^{O(\sqrt{\lambda})}$, i.e., the upper bound is at least $O(\sqrt{\lambda})$.  On the other hand, in \cite{DF, DF1} Donnelly and Fefferman showed that Laplace eigenfunctions can vanish at most of order $O(\sqrt{\lambda})$ on compact smooth Riemannian manifolds (Dirichlet/Neumann boundary conditions are needed with the presence of the boundaries). Naturally, one expects the vanishing order bound $O(\sqrt{\|V\|_{L^\infty}})$ for \eqref{eqn-211003-0419}. However, later Meshkov in \cite{MR1110071} discovered an example indicating that for complex potentials $V$ (and hence, complex solutions), the vanishing order can be as large as $C\|V\|_{L^\infty}^{2/3}$. Indeed, such order was proved to be a valid upper bound in \cite{MR2180453} for the complex case. Since the method in \cite{MR2180453} does not distinguish the real and complex cases, whether the order estimate $O(\sqrt{\|V\|_{L^\infty}})$ holds for real $V$ remains as an outstanding open question. 

	The study of the real case was initiated in \cite{K-DUKE} by Kukavica, who addressed the optimal vanishing rate of solutions to \eqref{eqn-211003-0419} with differentiable $V$. In this direction, the sharp upper bound $O(\|V\|_{W^{1,\infty}}^{1/2})$ was proved in \cite{MR3085618} and \cite{ZHUAJM} using different methods, which recovers Donnelly-Fefferman's result when $V=\lambda$. For merely bounded $V$, recently, remarkable progress was made in \cite{2020arXiv200707034L} on $\bR^2$. The authors proved the vanishing order estimate $O(\sqrt{\|V\|_{L^\infty}\log(\|V\|_{L^\infty})})$, which is sharp up to a log drift.

One could also study quantitative SUCP at a boundary point, if proper boundary conditions are given. Such problems with Dirichlet boundary conditions have been extensively studied in the literature, c.f. \cite{AE, CKW, MR1415331, MR1090434, 2020arXiv200410721T}. However, there are very few results on Neumann or Robin problems. In \cite{AE}, Adolfsson and Escauriaza proved that on $C^{1,1}$ domain, harmonic functions with zero Neumann boundary condition cannot vanish at infinite order ``from the bulk'' or from the boundary, i.e., for any $x_0\in\p\Omega$ and a non-trivial solution $u$, there exists a positive integer $N$, such that
\begin{align}
	\sup_{x\in B_r(x_0)\cap\Omega}|u(x)| \geq Cr^{N},\label{eqn-211003-0510-1}\\
	\sup_{x\in B_r(x_0)\cap\p\Omega}|u(x)| \geq Cr^{N}.\label{eqn-211003-0510-2}
\end{align}
In \cite{MR2162295}, Tao and Zhang addressed the SUCP \eqref{eqn-211003-0510-1} for \eqref{eqn-211003-0419} with a Lipschitz coefficient matrix $A$, a potential $V$ satisfying the so-called Kato-type condition, and the corresponding zero cornomal boundary condition. Moreover, in \cite{MR2162295}, a restrictive (geometric) condition was assumed: for any $x\in\p\Omega$ close to $x_0$,
\begin{equation}\label{eqn-211003-0511}
	A(x)(x-x_0)\cdot\vec{n}(x) = 0,
\end{equation}
where $\vec{n}$ is the outward unit normal to $\p\Omega$. Recently in \cite{DFV}, \eqref{eqn-211003-0510-1}-\eqref{eqn-211003-0510-2} were proved for \eqref{eqn-211003-0419}-\eqref{eqn-21003-0459} with $\Omega$ being a cone and $x_0$ being its vertex. The leading coefficient matrix was assumed to be $A=a(x)I_d$, which means the geometric condition \eqref{eqn-211003-0511} also holds for their problem.

On the other hand, in the spirit of \cite{DF, DF1}, a quantitative SUCP was proved in \cite{ZHU2018}. More precisely, for the so-called Robin eigenfunctions on smooth domains, i.e., solutions to \eqref{eqn-211003-0419}-\eqref{eqn-21003-0459} with $A=Id_d$ and $V=\lambda$, $\eta=\alpha$ both being constants, we have  \eqref{eqn-211003-0510-1} with $N= C(|\alpha| + \sqrt{|\lambda|})$. 

In the current paper, we generalize the results in \cite{AE} and \cite{ZHU2018} to operators with Lipschitz coefficients $A$ and Lipschitz potentials $V, \eta$. Our result also generalizes those in \cite{MR2162295} (and to some extent, \cite{DFV}), most importantly, by removing geometric condition \eqref{eqn-211003-0511} and proving a sharp vanishing order estimate.

Now we state our assumptions and main results. Consider the local problem
\begin{equation}\label{eqn-21-0109-1122}
	\begin{cases}
		\operatorname{div}(ADu) =  Vu  & \text{in }\, \Omega\cap B_2,\\
		ADu\cdot n = \eta u& \text{on }\, \p\Omega\cap B_2.
	\end{cases}
\end{equation}
We always assume $A=(a_{ij})$ to be elliptic and symmetric: for some $\lambda\in (0,1]$,
\begin{equation}\label{eqn-210706-0215}
	\lambda I_d \leq A \leq \lambda^{-1}I_d,\quad a_{ij}=a_{ji}.
\end{equation}
Furthermore, we assume the potentials satisfy
\begin{equation}\label{eqn-210912-0648}
	M:= \|V\|_{W^{1,\infty}} <\infty\quad M_\eta:=\|\eta\|_{W^{1,\infty}}<\infty.
\end{equation}
\begin{assumption}\label{ass-210912-0655}
	We assume that $a_{ij}\in W^{1,1}_{loc}(\Omega)$, and that for some $x_0\in\p\Omega$, there exists a positive function $\epsi=\epsi(r)$, such that for any $x\in\overline{\Omega}$ with $|x-x_0|\leq 1$,
	\begin{equation}\label{eqn-210606-1003}
		|Da_{ij}(x)|\leq \varepsilon(|x-x_0|)/|x-x_0|,\quad I_\epsi:=\int_0^1\epsi(r)/r\,dr<\infty.
	\end{equation}
Furthermore, we assume the conormal vector
	\begin{equation}\label{eqn-210903-0627}
		A\vec{n}\in \operatorname{Lip}(\p\Omega),
	\end{equation}
where $\vec{n}$ is the outward unit normal to $\p\Omega$.
\end{assumption}
Clearly, if $A$ is Lipschitz and $\p\Omega\in C^{1,1}$, then Assumption \ref{ass-210912-0655} is satisfied for any $x_0\in\p\Omega$. 
Our first result is the sharp vanishing order estimate, ``from the bulk''.
\begin{theorem}\label{thm-210208-0351}
	Let $\Omega\in C^{1,1}$ and $0\in\p\Omega$. Suppose that \eqref{eqn-210706-0215}, \eqref{eqn-210912-0648}, and Assumption \ref{ass-210912-0655} (with $x_0=0$) hold. Then there exist constants $R_0, r_0\in(0,1)$ depending on $(d,\lambda, I_\epsi, \|\p\Omega\|_{C^{1,1}}, \|A\vec{n}\|_{C^{0,1}})$, such that the following assertions hold.
	\begin{enumerate}
		\item  [(A)] For any nontrivial $W^{1,2}$-weak solution $u$ to \eqref{eqn-21-0109-1122} and any $r\leq R_0$,
		\begin{equation*}
			\big(\fint_{\Omega\cap B_r}|u|^2\big)^{1/2}\geq Cr^{C(\sqrt{M} + M_\eta + \log(\|u\|_{L^2(\Omega\cap B_1)}/\norm{u}_{L^2(\Omega\cap B_{r_0})}) + 1)}\big(\fint_{\Omega\cap B_{R_0}}|u|^2\big)^{1/2},
		\end{equation*}
		where $C$ depends on $\lambda,d, \|\p\Omega\|_{C^{1,1}}, \|A\vec{n}\|_{C^{0,1}}$, and $I_\epsi$.
		\item  [(B)] If we further assume that $\eta=-\eta_0<0$ is a constant, then
		\begin{equation*}
			\big(\fint_{\Omega\cap B_r}|u|^2\big)^{1/2}\geq Cr^{C(\sqrt{M} + \log(\|u\|_{L^2(\Omega\cap B_1)}/\norm{u}_{L^2(\Omega\cap B_{r_0})}) +1)}\big(\fint_{\Omega\cap B_{R_0}}|u|^2\big)^{1/2}.
		\end{equation*}
	\end{enumerate}		
\end{theorem}
For (A), our result is sharp even when $\eta$ and $V$ are constants, cf., \cite{ZHU2018}. The significance of the result in (B) is that the vanishing order is independent of $\eta_0$: When $\eta_0\rightarrow -\infty$, this recovers the estimate for Dirichlet problems. Whether such independence still holds for non-constant, negative $\eta$, i.e.,
$$\eta=\eta(x)\leq \eta_0<0$$
 remains open. Our next result is the unique continuation from the boundary.
\begin{theorem}\label{thm-210912-0701}
	Let $\Omega\in C^{1,1}$. Suppose that \eqref{eqn-210706-0215}-\eqref{eqn-210912-0648} hold and $a_{ij}$ is Lipschitz. Then any nontrivial $W^{1,2}$-weak solution $u$ to \eqref{eqn-21-0109-1122} cannot vanish at infinite order from the boundary. In other words, if for some $x_0\in\p\Omega\cap B_2$ and any $N>0$,
	\begin{equation}\label{eqn-210101-0254}
		u(x) = O(|x-x_0|^N),\quad \text{as}\,\,x (\in\p\Omega)\rightarrow x_0,
	\end{equation}
	then $u\equiv 0$.
\end{theorem}
\begin{corollary}\label{cor-211001-0138}
	Under conditions of Theorem \ref{thm-210912-0701}, any nontrivial $W^{1,2}$-weak solution $u$ to \eqref{eqn-21-0109-1122} cannot vanish on a subset of $\p\Omega\cap B_2$ with a positive surface measure.
\end{corollary}
\begin{remark}
	While finishing this paper, the authors noticed a recent nice result in \cite{2021arXiv211014282B} by Burq and Zuily, regarding the quantitative unique continuation for conormal problems, i.e., \eqref{eqn-21-0109-1122} with $\eta=0$. We address the Robin boundary value problem,  which is the main difference between \cite{2021arXiv211014282B} and the current paper. Essentially, after the flattening, they take the even extension for $u$ and proper extensions for $a_{ij}$ to the lower half space. Then, the original boundary becomes an interior surface, and the propogation of smallness result \cite{MR3966855} by Logunov and Malinnikova applies. However, such extension is not available in the Robin setting.
\end{remark}
\begin{remark}
	After the completion of the current paper, the authors learned from G. Alessandrini about the papers by Sincich \cite{MR2745797,MR3306378} in the context of inverse problems, where similar estimates were proved assuming the non-positivity of $\eta$.
\end{remark}
The key step in the current paper is to prove a sharp doubling inequality:
\begin{equation*}
	\fint_{\Omega\cap B_{2r}} |u|^2 \leq \left(\frac{\|u\|_{L^2(\Omega\cap B_1)}}{\norm{u}_{L^2(\Omega\cap B_{r_0})}}\right)^Ce^{C(\sqrt{M} + M_\eta+1)}\fint_{\Omega\cap B_{r}} |u|^2.
\end{equation*}
See \eqref{eqn-211206-0532}. For this, we prove the ``almost monotonicity'' of an Almgren-type frequency function. In a series of papers \cite{GL1, GL2}, Garofalo and Lin pioneered the use of such frequency functions to study the SUCP of elliptic equations. Furthermore, we adopt an idea in \cite{K-EJDE}: the weight function $(r^2-|x|^2)^\alpha$ is allowed in such computation (essentially, in the radial deformation). Optimizing $\alpha$ will lead us to sharp estimates. Similar ideas were also employed in \cite{ZHUAJM} and \cite{BG}.
    
    The rest of the paper is organized as follows. In Section \ref{sec-210706-0145}, we work on the half space. More precisely, we introduce the weighted frequency function on the half space, and prove the ``almost monotonicity'' in Proposition \ref{prop-210127-1144}, assuming the geometric condition \eqref{eqn-211003-0511}. Then in Section \ref{sec-flat}, we construct a change of variable which reduces the problem to the half-space case. From these, the doubling inequality in Proposition \ref{prop-211224-0522} is derived in Section \ref{sec-doubling}. Using the doubling inequality, we prove Theorem \ref{thm-210208-0351} (A) in Section \ref{sec-doubling} and Theorem \ref{thm-210208-0351} (B) in Section \ref{sec-thm1b}. Eventually, Theorem \ref{thm-210912-0701} and Corollary \ref{cor-211001-0138} are proved in Section \ref{sec-thm2}.
    
	\section{Frequency function and monotonicity on half space}\label{sec-210706-0145}
	In this section, we work on the half space $\bR^d_+ =\{(x',x_d): x_d>0\}$. Let us denote
	\begin{equation*}
		B_r^{+} = B_r(0)\cap\bR^{d}_{+},\quad \Gamma_r = B_r(0)\cap\p\bR^{d}_{+}.
	\end{equation*} We consider the problem
	\begin{equation}\label{eqn-21-0706-0210}
	\begin{cases}
	\operatorname{div}(ADu) =  Vu  & \text{in }\, B_2^+,\\
	ADu\cdot \vec{n} = \eta u& \text{on }\, \Gamma_2.
	\end{cases}
	\end{equation}
	Throughout this section, we assume \eqref{eqn-210706-0215}-\eqref{eqn-210912-0648}, Assumption \ref{ass-210912-0655} with $\Omega = \bR^d_+$ and $x_0=0$, and the matrix $A$ has a special shape
	\begin{equation}\label{e.w11301}
	A=\begin{pmatrix}
	\tilde{A}&0\\
	0&a_{dd}
	\end{pmatrix}\quad \text{and}\quad A(0)=I_d.
	\end{equation}
One could simply check
	\begin{equation*}
	\left<A(x)x,\vec{n}(x)\right>=0\quad\forall x\in\Gamma_2.
	\end{equation*}
	As in \cite{BG, GL1, GL2}, we define the conformal factor
	\begin{equation}\label{eqn-211114-0749}
	\mu(x)=\frac{\langle Ax, x\rangle}{|x|^2}.
	\end{equation}
	For $\alpha \geq 1$  and $r\in(0,2)$, we define the weighted height, energy, and frequency functions as
	\begin{equation}\label{eqn-210606-1036}
	\begin{split}
	H(r) &:= \int_{B_r^{+}} |u(x)|^{2}(r^{2}-|x|^{2})^{\alpha}\mu(x)\,dx,\\
	I(r) &:= 2(\alpha+1)\int_{B_{r}^{+}} \left<A(x)Du(x), x\right> u(x)(r^{2}-|x|^{2})^{\alpha}\,dx,\\
	N(r) &:= \frac{I(r)}{H(r)}.
	\end{split}
	\end{equation}
Now we state the main results in this section: the almost monotonicity of $N$.
	\begin{proposition}\label{prop-210127-1144}
		For any $r\in(0,1]$, we have
		\begin{equation}\label{eqn-210628-0104}
		N'(r) \geq -
		C\frac{\tepsi(r)}{r}N(r) -
		C\frac{\tepsi(r)}{r}(Mr + \alpha + M_\eta^2),
		\end{equation}
	where $\tepsi(r):=\epsi(r)+r$ and $C=C(d,\lambda)$.
	\end{proposition}
From Proposition \ref{prop-210127-1144}, we have the doubling property of $H$.
\begin{lemma}\label{lem-211211-1209}
	There exists a constant $\delta_0=\delta_0(d,\lambda, I_\epsi)\in(0,1)$, such that the following assertion holds. For any $\kappa>1$, $R\in (0,1]$ and $r\in(0,R/\kappa)$, there exists a constant $\delta \in (\delta_0,1)$ depending on $(d,\lambda,I_\epsi,\kappa,r,R)$, such that
	\begin{equation*}
		\log\left(\frac{H(\kappa r)}{H(r)}\right) \leq C\log\left(\frac{H(R)}{H(R/\kappa^{\delta})}\right) + C(\alpha+1)^{-1}(M+\alpha+\alpha^2+M_\eta^2).
	\end{equation*}
Here $C=C(d,\lambda, I_\epsi,\kappa)$.
\end{lemma}
The rest of this section will be devoted to the proofs of Proposition \ref{prop-210127-1144} and Lemma \ref{lem-211211-1209}. To compute $N'$, we need $H'$ and $I'$, which are related to the first and second variations of $H$. The computation of $H'$ is straightforward, which can be found in Section \ref{sec-211028-0625-1}. The computation of $I'$ is in Section \ref{sec-210620-0103}, which is based on the weighted inequalities obtained in Section \ref{sec-211127-1031}. Eventually, we prove Proposition \ref{prop-210127-1144} and Lemma \ref{lem-211211-1209} in Section \ref{sec-210714-1249}.

\subsection{Some preliminaries}
Testing \eqref{eqn-21-0109-1122} by $u(x)(r^{2}-|x|^{2})^{\alpha+1}$ and doing the integration by parts, we can rewrite $I$ in an equivalent form
\begin{equation}\label{eqn-210606-1039}
	\begin{split}
		I(r) 
		&=
		\int_{B_{r}^{+}} \left<ADu,Du\right>(r^{2}-|x|^{2})^{\alpha+1} +\int_{B_{r}^{+}} V|u|^{2}(r^{2}-|x|^{2})^{\alpha+1} 
		\\&\quad- \int_{\Gamma_r}\eta |u|^2(r^{2}-|x|^{2})^{\alpha+1}
		\\&= I_1 + I_2 + I_3
		.
	\end{split}
\end{equation}
Clearly $I_1>0$, while $I_2$ and $I_3$ might be sign-changing. We define the majorants
\begin{equation}\label{eqn-210713-1102}
\It_2:=\int_{B_{r}^{+}} |V||u|^{2}(r^{2}-|x|^{2})^{\alpha+1},\quad \It_3 :=  \int_{\Gamma_r}|\eta||u|^2(r^{2}-|x|^{2})^{\alpha+1}.
\end{equation}
	Next, we discuss some properties of $\mu$ defined in \eqref{eqn-211114-0749}. First of all, from \eqref{eqn-210706-0215} and the fact that $A(0)=I_d$,
	\begin{equation}\label{eqn-210615-1129-1}
		\lambda \leq \mu \leq \lambda^{-1},
		\quad
		\mu(0)=1.
	\end{equation}
	From \eqref{eqn-210606-1003}, we see
	\begin{equation}\label{eqn-210615-1129-2}
		|\mu(x) - \mu(0)|\leq C\epsi(|x|).
	\end{equation}
	Furthermore, we have
	\begin{equation}\label{eqn-210616-1227}
	\begin{split}
	&|D_i \mu|
	=
	\left|D_{i}\left(\frac{(a_{jk}-\delta_{jk})x_{k}x_{j}}{|x|^2}\right)\right|
	\\&=
	\left|\frac{D_{i}(a_{jk}-\delta_{jk})x_{k}x_{j}+ (a_{jk}-\delta_{jk})D_{i}(x_{k}x_{j}) }{|x|^2}
	-
	2\frac{(a_{jk}-\delta_{jk})x_{k}x_{j}x_{i}}{|x|^4}\right|
	\leq
	C(d)\frac{\varepsilon(|x|)}{|x|}
	.
	\end{split}
	\end{equation}
	The following vector field is used in our computation:
	$$\beta:=(\beta_1(x),\cdots,\beta_d(x))^T=Ax/\mu.$$
	\begin{lemma}\label{lem-210620-0103}
		For $\beta$ defined above, the following hold
		\begin{align}
		&\beta(x)\cdot x
		=
		|x|^2,\label{eqn-210620-0121-1}
		\\&|x|\leq |\beta(x)| \leq \lambda^{-2}|x|\label{eqn-210620-0121-2}
		,
		\\&|D_i\beta_k - \delta_{ik}| 
		\leq 
		C(d,\lambda)\epsi(|x|).\label{eqn-210616-1252-1}
		\end{align}
	\end{lemma} 
\begin{proof}
	By the definitions of $\beta$ and $\mu$,
	\begin{equation*}
	\beta\cdot x = \frac{Ax}{\mu}\cdot x = \frac{Ax\cdot x}{(Ax\cdot x)/|x|^2} = |x|^2,
	\end{equation*}
   which is \eqref{eqn-210620-0121-1}. From this, the lower bound in \eqref{eqn-210620-0121-2} follows simply by the Cauchy inequality. For the upper bound in \eqref{eqn-210620-0121-2}, from \eqref{eqn-210706-0215}, we have
   \begin{equation*}
   	|\beta| = \frac{|Ax|}{(Ax\cdot x)/|x|^2} \leq \frac{\lambda^{-1}|x|}{\lambda} = \lambda^{-2}|x|.
\end{equation*}
The computation for derivatives in \eqref{eqn-210616-1252-1} requires some more work:
\begin{equation}
	\begin{split}\label{e.w11302}
	D_{i}\beta_{k}
	=
	D_{i}\left(\frac{a_{kl}x_{l}}{\mu}\right)
	&=
	\frac{(D_{i}a_{kl})x_{l}}{\mu}
	+
	\frac{a_{kl}\delta_{il}}{\mu}
	-
	\frac{a_{kl}x_{l}D_{i}\mu}{\mu^2}
	\\&=
	\frac{(D_{i}a_{kl})x_{l}}{\mu}
	+
	\frac{(a_{kl}-\delta_{kl})\delta_{il}}{\mu(0)}
	+
	\frac{\delta_{kl}\delta_{il}}{\mu(0)}
	+
	\delta_{kl}\delta_{il}\left(\frac{1}{\mu}-\frac{1}{\mu(0)}\right)
	-
	\frac{a_{kl}x_{l}D_{i}\mu}{\mu^2}
	.
	\end{split}
	\end{equation}
	From this, we obtain \eqref{eqn-210616-1252-1} by noting $\mu(0)=1$, \eqref{eqn-210606-1003}, \eqref{eqn-210615-1129-1}, \eqref{eqn-210615-1129-2}, \eqref{eqn-210616-1227}, and \eqref{e.w11302}.
\end{proof}

	\subsection{First variation of $H$}\label{sec-211028-0625-1}
	Recall the definitions of $H, I$ in \eqref{eqn-210606-1036}. In this section, we compute $H'$.
	\begin{lemma}\label{lem-210615-1157}
		Suppose that $u\in W^{1,2}(B_2^+)$ solves \eqref{eqn-21-0706-0210}. Then for any $r\in (0,2)$,
		\begin{equation*}
			H'(r) = \frac{2\alpha + d +O(1)\epsi(r)}{r}H(r) + \frac{I(r)}{(\alpha+1)r}.
		\end{equation*}
	Here and throughout the section, $O(1)$ represents a function which is bounded as $r\rightarrow 0$.
	\end{lemma}
	\begin{proof}
		Recall the definition of $H$ in \eqref{eqn-210606-1036},
		\begin{align*}\label{eqn-210128-0218}
			H' 
			&=	
			2\alpha\int_{B_r^+}r u^{2}(r^{2}-|x|^{2})^{\alpha-1}\mu
			=	
			2\alpha\int_{B_r^+}\left(\frac{r^2 - |x|^2}{r}+\frac{|x|^2}{r}\right) u^{2}(r^{2}-|x|^{2})^{\alpha-1}\frac{\langle Ax, x\rangle}{|x|^2}
			\\&=
			\frac{2\alpha }{r}H(r) +  \frac{2\alpha }{r}\int_{B_r^+} \langle Ax, x\rangle u^2(r^2-|x|^2)^{\alpha-1}	
			=
			\frac{2\alpha }{r}H(r) -  \frac{2\alpha }{r}\int_{B_r^+} Ax u^2\frac{D(r^2-|x|^2)^{\alpha}	}{2\alpha}
			\\&=
			\frac{2\alpha }{r}H(r) -  \frac{1 }{r}\int_{B_r^+} Ax u^2 D(r^2-|x|^2)^{\alpha}
			.
		\end{align*}
		For the second term, we do integration by parts. Noting that 
		\begin{equation*}
			\p B_r^+ = \Gamma_r \cup \{\p B_r\cap\bR^d_+\}
		\end{equation*}
		and
		\begin{equation*}
			Ax\cdot \vec{n}=0\,\,\text{on}\,\,\Gamma_r,\quad  (r^2-|x|^2)^{\alpha}=0 \,\,\text{on}\,\,\p B_r,
		\end{equation*}
		we have
		\begin{align*}
			-  \frac{1 }{r}\int_{B_r^+} Ax u^2 D(r^2-|x|^2)^{\alpha}
			&=
			-\frac{1 }{r}\int_{\p B_r^+} u^2(Ax\cdot n) (r^2-|x|^2)^{\alpha}
			+
			\frac{1 }{r}\int_{B_r^+} \text{div}(u^2Ax) (r^2-|x|^2)^{\alpha}
			\\&=
			\frac{1 }{r}\int_{B_r^+} 2u\langle ADu, x\rangle (r^2-|x|^2)^{\alpha}
			+
			\frac{1 }{r}\int_{B_r^+} u^2\text{div}(Ax ) (r^2-|x|^2)^{\alpha}
			\\&=
			\frac{I(r)}{(\alpha+1)r}
			+
			\frac{1 }{r}\int_{B_r^+} u^2\partial_{i}(a_{ij}x_{j}) (r^2-|x|^2)^{\alpha}
			\\&=
			\frac{I(r)}{(\alpha+1)r}
			+
			\frac{1 }{r}\int_{B_r^+} u^2\big(\operatorname{tr}(I_d) +\operatorname{tr}(A-I_d) + x_{j}\partial_{i}a_{ij}\big) (r^2-|x|^2)^{\alpha}
			\\&=
			\frac{I(r)}{(\alpha+1)r}+\frac{d}{r}H(r)
			\\&\quad+\frac{1 }{r}\int_{B_r^+} u^2\big(d (\mu(0)-\mu) +\operatorname{tr}(A-I_d) + x_{j}\partial_{i}a_{ij}\big) (r^2-|x|^2)^{\alpha}
			.
		\end{align*}
		By \eqref{eqn-210606-1003}, \eqref{eqn-210615-1129-1}, and \eqref{eqn-210615-1129-2}, we get
		\begin{align*}
			|d (\mu(0)-\mu) +\operatorname{tr}(A-I_d) + x_{j}\partial_{i}a_{ij}| \leq C(d,\lambda)\epsi(r)\mu.
		\end{align*}
		Combining all above yields the desired conclusion.
	\end{proof}

	\subsection{Some weighted inequalities}\label{sec-211127-1031}
In this section, we prove some weighted inequalities which are needed in the computation of $I'$. Throughout section \ref{sec-211127-1031}, we do not use the equation -- the inequalities hold for any Sobolev functions $u$.

Recall the definition of $I_1$ in \eqref{eqn-210606-1039}.
\begin{lemma}\label{lem-210627-1026}
	For any $u\in W^{1,2}(B_r^+)$ and $\alpha\geq 1$,
	\begin{equation*}
	\alpha^2\int_{B_r^{+}}u^2 (r^2-|x|^2)^{\alpha-1}|x|^2 \leq C(\alpha H(r) + I_1(r))
	,
	\end{equation*}
	where $C=C(d,\lambda)$.
\end{lemma}
\begin{proof}
	By the divergence theorem, we see
	\begin{align}
	\alpha^{2}
	\int_{B_r^{+}}u^2 (r^2-|x|^2)^{\alpha-1}|x|^2 
	&=
	-
	\alpha/2
	\int_{B_r^{+}}u^2 x\cdot \nabla(r^2-|x|^2)^{\alpha}\nonumber
	=
	\alpha/2
	\int_{B_r^{+}}\text{div}(u^2 x) (r^2-|x|^2)^{\alpha}\nonumber
	\\&=
	d\alpha/2
	\int_{B_r^{+}}u^2 (r^2-|x|^2)^{\alpha}
	+
	\alpha
	\int_{B_r^{+}}u (x\cdot \nabla u) (r^2-|x|^2)^{\alpha}\nonumber
	.
\end{align}
Furthermore,
\begin{align}
&\alpha
\int_{B_r^{+}}u (x\cdot \nabla u) (r^2-|x|^2)^{\alpha}
=
\alpha
\int_{B_r^{+}}u (x\cdot \nabla u) \left((r^2-|x|^2)^{\alpha+1}(r^2-|x|^2)^{\alpha-1}\right)^{\frac{1}{2}}\nonumber
	\\&\leq
	\frac{\alpha^2}{2}
	\int_{B_r^{+}}u^2 (r^2-|x|^2)^{\alpha-1}|x|^2
	+
	\frac{1}{2} \int_{B_r^{+}}|\nabla u|^{2} (r^2-|x|^2)^{\alpha+1}\label{eqn-210627-0947}
	\\&\leq
	\frac{\alpha^2}{2}
	\int_{B_r^{+}}u^2 (r^2-|x|^2)^{\alpha-1}|x|^2 
	+
	\frac{\lambda^{-1}}{2} I_{1}
	,\label{eqn-211009-1139}
	\end{align}
	where in \eqref{eqn-210627-0947} we applied Young's inequality. Absorbing the first term in \eqref{eqn-211009-1139} and noting \eqref{eqn-210615-1129-1}, the lemma is proved.
\end{proof}		
In order to bound the boundary terms on $\Gamma_r$, we need the following interpolation-type trace inequality.
\begin{lemma}\label{lem-210620-1003}
	For any $\alpha\geq 1$, $\delta>0$, and $u\in W^{1,2}(B_r^+)$ with $r\leq 1$, we have
	\begin{equation*}
	\int_{\Gamma_r} |u|^2 (r^2-|x|^2)^{\alpha+1} \leq C(\delta I_1 + \delta\alpha H + \delta^{-1}r^2H),
	\end{equation*}
	where $C=C(d,\lambda)$.
\end{lemma} 
\begin{proof}
	We first prove that for any $v\in C^\infty_c(\overline{\bR^d_+})$,
	\begin{equation}\label{eqn-210627-0952}
	\|v(\cdot,0)\|_{L^2(\bR^{d-1})}^2 \leq \delta\|Dv\|_{L^2(\bR^{d}_{+})}^2 + \frac{1}{\delta}\|v\|_{L^2(\bR^d_{+})}^2.
	\end{equation}
	Indeed, we see
	\begin{align}
	\|v(\cdot,0)\|_{L^2(\bR^{d-1})}^2 
	=
	\int_{\mathbb R^{d-1}} |v(x',0)|^{2}\,dx'
	&=
	\int_{\mathbb R^{d-1}}\int_{0}^{\infty} \frac{\partial |v(x',x_{d})|^{2}}{\partial x_{d}}\,dx_{d}\,dx'\nonumber
	\\&=
	2\int_{\mathbb R^{d-1}}\int_{0}^{\infty} \frac{\partial v(x',x_{d})}{\partial x_{d}} v(x',x_{d})\,dx_{d}\,dx'\nonumber
	\\
	&\leq
	\delta\|Dv\|_{L^2(\bR^d_+)}^2 + \frac{1}{\delta}\|v\|_{L^2(\bR^d_+)}^2.\label{eqn-210627-1005}
	\end{align}
	Here in \eqref{eqn-210627-1005},  we applied Young's inequality. Hence, \eqref{eqn-210627-0952} is proved. Now, (by density argument), we apply \eqref{eqn-210627-0952} with
	$$v=u(r^2 - |x|^2)^{(\alpha+1)/2}1_{B_r^+}\in W^{1,2}(\bR^d_+),$$
	which yields
	\begin{align}
	\int_{\Gamma_r} |u|^{2}(r^2 - |x|^2)^{\alpha+1}
	&\leq
	\delta \left(\int_{B_r^+} |Du|^{2}(r^2 - |x|^2)^{\alpha+1}
	+
	(\alpha+1)^{2}\int_{B_r^+} |x|^{2}u^{2}(r^2 - |x|^2)^{\alpha-1}\right)\nonumber
	\\&\quad+
	\frac{1}{\delta}\int_{B_r^+} u^{2}(r^2 - |x|^2)^{\alpha+1}
	\nonumber
	\leq
	\frac{\delta I_1}{\lambda} + C\delta\frac{(\alpha+1)^2}{\alpha^2}(\alpha H + I_1) + \frac{ r^2}{\lambda \delta}H\label{eqn-210627-1022}
	,
	\end{align}
	where we used Lemma \ref{lem-210627-1026}, \eqref{eqn-210615-1129-1}, and \eqref{eqn-210606-1036}. Noting $\alpha\geq 1$, the lemma is proved.
\end{proof}
As an application, we provide some estimates on the majorants. Recall the definitions of $\It_2$ and $\It_3$ in \ref{eqn-210713-1102}.
\begin{lemma}\label{lem-210620-1120}
	For any $u\in W^{1,2}(B_r^+)$ with $r\leq 1$ and $\alpha\geq 1$,
	\begin{equation}\label{eqn-210620-1114-1}
	\It_2\leq M\lambda^{-1} r^2H,
	\end{equation}
	\begin{equation}\label{eqn-210620-1114-2}
	\It_3 + \int_{\Gamma_r} u^2|\beta\cdot D_{T}\eta|(r^2-|x|^2)^{\alpha+1} \leq \frac{r}{2}I_1 + \frac{\alpha r}{2} H + C M_\eta^2rH,
	\end{equation}
	where $D_T=(D_{x'},0)$ is the tangential gradient, and $C=C(d,\lambda)$ is a constant.
\end{lemma}
\begin{proof}
	The estimate \eqref{eqn-210620-1114-1} follows directly from the definition of $\tilde{I}_{2}$, \eqref{eqn-210912-0648}, and \eqref{eqn-210615-1129-1}. As for \eqref{eqn-210620-1114-2}, we first derive from \eqref{eqn-210912-0648} and \eqref{eqn-210620-0121-2} that
	$$|\eta|\leq M_\eta\quad \text{and}\quad |\beta\cdot D_T\eta|\leq |\beta||D\eta|\leq  \lambda^{-2}rM_\eta.$$ 
	Combining this and the fact that $r\leq 1$, we have
	\begin{align*}
			\It_3 + \int_{\Gamma_r} |u|^2|\beta\cdot D_{T}\eta|(r^2-|x|^2)^{\alpha+1} \leq M_\eta(1+\lambda^{-2}) \int_{\Gamma_r}|u|^2(r^2-|x|^2)^{\alpha+1},
		\end{align*}
	from which \eqref{eqn-210620-1114-2} can be obtained by applying Lemma \ref{lem-210620-1003} with 
	$$
	\delta=C^{-1}r(1+\lambda^{-2})^{-1}M_\eta^{-1}/2.
	$$
\end{proof}
\begin{lemma}\label{lem-210621-1219}
	For any $\alpha\geq 1$ and $u\in W^{1,2}(B_r^+)$ with $r\leq 1$,
	\begin{equation*}
	I_1(r) \leq 2I(r) + C(Mr^2 + \alpha r + M_\eta^2r)H(r),
	\end{equation*}
where $C=C(d,\lambda)$.
\end{lemma}
\begin{proof}
	By direct computation, Lemma \ref{lem-210620-1120}, and $r\leq 1$, we have
	\begin{equation*}
	I_1=I - I_2 - I_3 \leq I + \It_2 + \It_3 \leq I + M\lambda^{-1} r^2H + \frac{1}{2}I_1 + Cr(\alpha + M_\eta^2)H.
	\end{equation*}
	From this, the lemma can be proved by absorbing $I_1/2$ on the right-hand side.
\end{proof}

	\subsection{Computing the second variation}\label{sec-210620-0103}
	The key step in proving Proposition \ref{prop-210127-1144} is the following   lower bound of $I'$, which combined with Lemma \ref{lem-210615-1157}, gives the second variation of $H$. 
	\begin{lemma}\label{lem-210628-1214}
		Suppose that $u\in W^{1,2}(B_2^+)$ solves \eqref{eqn-21-0706-0210}. Then for any $r\in(0,1]$, we have
				\begin{equation*}
		I' \geq \frac{d+2\alpha}{r}I
		-
		C\frac{\tepsi(r)}{r}I
		-
		C\frac{\tepsi(r)}{r}(Mr + \alpha + M_\eta^2)H
		+
		\frac{4(\alpha+1)}{r}\int_{B_r^+}(\beta\cdot Du)(ADu\cdot x) (r^{2}-|x|^{2})^{\alpha}\,dx.
		\end{equation*}
	Here, $C=C(d,\lambda)$.
	\end{lemma}
	\begin{proof}
		Recall 
		\begin{equation*}
			\begin{split}
				I(r) 
				&=
				\int_{B_{r}^{+}} \left<ADu,Du\right>(r^{2}-|x|^{2})^{\alpha+1} +\int_{B_{r}^{+}} V|u|^{2}(r^{2}-|x|^{2})^{\alpha+1} - \int_{\Gamma_r}\eta |u|^2(r^{2}-|x|^{2})^{\alpha+1}
				\\&=
				I_{1}+I_{2}+I_{3}
				.
			\end{split}
		\end{equation*}
		From direct computation, we have
		\begin{equation}\label{eqn-210620-1029-1}
			I_2' = 2(\alpha+1)r\int_{B_{r}^{+}} V|u|^{2}(r^{2}-|x|^{2})^{\alpha}
		\end{equation}
		and
		\begin{equation}\label{eqn-210620-1029-2}
			I_3' = -2(\alpha+1)r\int_{\Gamma_r}\eta|u|^2(r^2-|x|^2)^{\alpha}.
		\end{equation}
The rest of this section is devoted to the computation of $I_1'$. To begin with,
		\begin{equation}
		\begin{split}
				I_{1}'
				&=
				2(\alpha+1)\int_{B_r^+}\left<ADu,Du\right>\frac{r^2 -|x|^2+|x|^2}{r}(r^{2}-|x|^{2})^{\alpha}\,dx\nonumber
				\\&=
				\frac{2(\alpha+1)}{r}\int_{B_r^+}\left<ADu,Du\right>(r^{2}-|x|^{2})^{\alpha+1}\,dx
				\\&\quad+
				\frac{2(\alpha+1)}{r}\int_{B_r^+}\left<ADu,Du\right>(Ax\cdot x)\frac{|x|^2}{Ax\cdot x}(r^{2}-|x|^{2})^{\alpha}\,dx\nonumber
				\\&=
				\frac{2(\alpha+1)}{r}I_{1}
				-
				\frac{1}{r}\int_{B_r^+}\left<ADu,Du\right> \langle Ax, D((r^{2}-|x|^{2})^{\alpha+1})\rangle\frac{1}{\mu}\,dx
				.\label{eqn-210616-1257-1}
		\end{split}
		\end{equation}
		By divergence theorem, we have
		\begin{align}
				I_1'&=
				\frac{2(\alpha+1)}{r}I_{1}
				-
				\frac{1}{r}\int_{\Gamma_r}\langle ADu, Du\rangle (Ax\cdot n) (r^{2}-|x|^{2})^{\alpha+1}\frac{1}{\mu}\,dx\nonumber
				\\&\quad+
				\frac{1}{r}\int_{B_r^+}\text{div}\left(\langle ADu, Du\rangle\frac{1}{\mu}Ax\right) (r^{2}-|x|^{2})^{\alpha+1}\,dx\nonumber
				\\&=
				\frac{2(\alpha+1)}{r}I_{1}
				+
				\frac{1}{r}\int_{B_r^+}\text{div}\left(\langle ADu, Du\rangle\frac{1}{\mu}Ax\right) (r^{2}-|x|^{2})^{\alpha+1}\,dx.\label{eqn-210616-1257-2}
		\end{align}
		Here in \eqref{eqn-210616-1257-2} we used the fact that $Ax\cdot \vec{n}=0$ on $\Gamma_r$. Next, we apply the following generalized Rellich's identity from \cite{PW}
		\begin{equation}\label{eqn-210616-1217}
		\begin{split}
		\text{div}(\langle ADu, Du\rangle \beta)
		&=
		2\text{div}(\langle\beta, Du\rangle ADu )
		+
		\text{div}(\beta) \langle ADu, Du\rangle
		-
		2(D_{i}\beta_{k})a_{ij}D_{j}uD_{k}u
		\\&\quad-
		2\langle \beta, Du\rangle \text{div}(ADu)
		+
		\beta_{k}(D_{k}a_{ij})D_{i}uD_{j}u
		,
		\end{split}
		\end{equation}
	with the vector field
	$$
	\beta:=(\beta_1(x),\cdots,\beta_d(x))^T=Ax/\mu.
	$$
		By \eqref{eqn-210606-1003} and \eqref{eqn-210620-0121-2}, we get
		\begin{equation}\label{eqn-210616-1252-3}
			|\beta_{k}(D_{k}a_{ij})D_{i}uD_{j}u|
			\leq
			C(d,\lambda)\epsi(r)|Du|^2
			\leq
			C(d,\lambda)\epsi(r)\langle ADu,Du\rangle
			.
		\end{equation}
	From \eqref{eqn-210616-1252-1}, we have
	\begin{equation}\label{eqn-210616-1252-2}
		\text{div}~\beta
		=
		d+\varepsilon(r)O(1).
\end{equation}
Substituting \eqref{eqn-210616-1252-1}, \eqref{eqn-210616-1252-3}, and \eqref{eqn-210616-1252-2} back to \eqref{eqn-210616-1217}, we have
		\begin{equation*}
		\begin{split}
			\text{div}(\langle ADu, Du\rangle \beta)
			&=
			2\text{div}(\langle\beta, Du\rangle ADu )
			+
			\big(d-2+\varepsilon(r)O(1)\big) \langle ADu, Du\rangle
			\\&\quad-
			2\langle\beta, Du\rangle Vu
			,
			\end{split}
		\end{equation*}
		where we used the equation $\operatorname{div}(ADu)=Vu$. Hence,
		\begin{align}
				&\frac{1}{r}\int_{B_r^+}\text{div}(\langle ADu, Du\rangle \beta) (r^{2}-|x|^{2})^{\alpha+1}\,dx\nonumber
				=
				\frac{2}{r}\int_{B_r^+}\text{div}(\langle\beta, Du\rangle ADu ) (r^{2}-|x|^{2})^{\alpha+1}\,dx\nonumber
				\\&\quad+
				\frac{(d-2+\varepsilon(r)O(1))}{r}\int_{B_r^+} \left<ADu, Du\right> (r^{2}-|x|^{2})^{\alpha+1}\,dx\nonumber
				\\&\quad-
				\frac{2}{r}\int_{B_r^+}(\beta\cdot Du)Vu (r^{2}-|x|^{2})^{\alpha+1}\,dx
				=:
				I_{11}+I_{12}+I_{13}\label{eqn-210714-1201}
				.
		\end{align}
		For $I_{11}$, we see the boundary condition $ADu\cdot n = \eta u$ on $\Gamma_r$ and
		\begin{equation*}
		\beta\cdot Du = \frac{Ax\cdot Du}{\mu} = \frac{\left<A^TDu,x\right>}{\mu} = \frac{\left<ADu,x\right>}{\mu}
		\end{equation*}
		in \eqref{eqn-219713-1159}.
		Notice that $$\beta\cdot n = (Ax\cdot n)/\mu=0\quad\text{on}\,\,\Gamma_r,$$
		then we have
		\begin{equation*}
		\beta\cdot D(u^2) = \beta\cdot D_T(u^2) = \sum_{i=1}^{d-1}\beta_i D_i(u^2).
		\end{equation*}
		Thus, we apply the divergence theorem and get
		\begin{align}
				&I_{11}
				=
				\frac{2}{r}\int_{\Gamma_{r}}\langle\beta, Du\rangle (ADu\cdot n) (r^{2}-|x|^{2})^{\alpha+1}\,dx'
				-
				\frac{2}{r}\int_{B_r^+}\langle\beta,Du\rangle \langle ADu, D( (r^{2}-|x|^{2})^{\alpha+1})\rangle\,dx\nonumber
				\\&=
				\frac{2}{r}\int_{\Gamma_{r}}(\beta\cdot Du)\eta u (r^{2}-|x|^{2})^{\alpha+1}\,dx'
				+
				\frac{4(\alpha+1)}{r}\int_{B_r^+}(\beta\cdot Du)(ADu\cdot x) (r^{2}-|x|^{2})^{\alpha}\,dx\nonumber
				\\&=
				\frac{1}{r}\int_{\Gamma_{r}}(\beta\cdot D(u^2))\eta (r^{2}-|x|^{2})^{\alpha+1}\,dx'
				+
				\frac{4(\alpha+1)}{r}\int_{B_r^+}(ADu\cdot x)^2\mu^{-1} (r^{2}-|x|^{2})^{\alpha}\,dx\label{eqn-219713-1159}
				.
		\end{align}
Recall that $D_T$ is the tangential gradient operator. Now we integrate by part on $\Gamma_r$ and obtain
\begin{align}
	&\frac{1}{r}\int_{\Gamma_{r}}(\beta\cdot D(u^2))\eta (r^{2}-|x|^{2})^{\alpha+1}\,dx'\nonumber
	=
	-\frac{1}{r}\int_{\Gamma_r} u^2 \sum_{i=1}^{d-1}D_i\big(\eta (r^2-|x|^2)^{\alpha+1}\beta_{i}\big)\nonumber
	\\&=
	-\frac{1}{r}\int_{\Gamma_r} u^2(\beta\cdot D_{T}\eta)(r^2-|x|^2)^{\alpha+1} + \frac{2(\alpha+1)}{r}\int_{\Gamma_r}u^2\eta(r^2-|x|^2)^{\alpha}(\beta\cdot x) \nonumber
	\\&\quad-
	\frac{1}{r}\int_{\Gamma_r}u^2\eta(r^2-|x|^2)^{\alpha+1}\sum_{i=1}^{d-1}D_i\beta_i
	=I_{111}+I_{112}+I_{113}.\label{eqn-210713-1115}
\end{align}
We compute term by term. From Lemma \ref{lem-210620-1120},
\begin{equation}\label{e.w11161}
	|I_{111}| \leq C\big(I_1 + \alpha H + M_\eta^2H\big)
\end{equation}
Next, noting \eqref{eqn-210620-0121-1}, we have
\begin{equation}\label{e.w11162}
 \begin{split}
	I_{112}
	&=
	\frac{2(\alpha+1)}{r}\int_{\Gamma_r}u^2\eta(r^2-|x|^2)^{\alpha}(|x|^2-r^2+r^2)
	\\&=
	\frac{2(\alpha+1)}{r}I_3 + 2(\alpha+1)r\int_{\Gamma_r}u^2\eta(r^2-|x|^2)^{\alpha}.
 \end{split}
\end{equation}
Last, by \eqref{eqn-210616-1252-1},
\begin{equation}\label{eqn-210714-1207-2}
	I_{113}
	=
	\frac{d-1}{r}I_3+\frac{\epsi(r)O(1)}{r}\It_3.
\end{equation}
Substituting \eqref{e.w11161}-\eqref{eqn-210714-1207-2} back to \eqref{eqn-210713-1115}, and then \eqref{eqn-219713-1159}, from \eqref{eqn-219713-1159} we have
\begin{equation}\label{eqn-210627-1156-1}
	\begin{split}
	I_{11}
	&\geq
	-C\big(I_1 + \alpha H + M_\eta^2H\big)
	+
	\frac{d+2\alpha+1}{r}I_3
	-C\frac{\epsi(r)}{r}\It_3
	\\&\quad+
	2(\alpha+1)r\int_{\Gamma_r}u^2\eta(r^2-|x|^2)^{\alpha}
	+
	\frac{4(\alpha+1)}{r}\int_{B_r^+}(ADu\cdot x)^2\mu^{-1} (r^{2}-|x|^{2})^{\alpha}\,dx
	\end{split}
\end{equation}
	For $I_{12}$ in \eqref{eqn-210714-1201}, by the definition of $I_1$ in \eqref{eqn-210606-1039}, we obtain
		\begin{equation}\label{eqn-210627-1156-2}
				I_{12}
				=
				\frac{d-2+\varepsilon(r)O(1)}{r}I_{1}
				.
		\end{equation}
		For $I_{13}$ in \eqref{eqn-210714-1201}, we do integration by parts again.
		\begin{align}
				I_{13}
				&=
				-	
				\frac{1}{r}\int_{B_r^+}\langle \beta, D(u^{2})\rangle V (r^{2}-|x|^{2})^{\alpha+1}\,dx\nonumber
				\\&=
				-
				\frac{1}{r}\int_{\Gamma_r}(\beta\cdot n)Vu^{2} (r^{2}-|x|^{2})^{\alpha+1}\,dx'
				+
				\frac{1}{r}\int_{B_r^+}\operatorname{div}(V (r^{2}-|x|^{2})^{\alpha+1}\beta)u^{2}\,dx\nonumber
				\\&=
				\frac{1}{r}\int_{B_r^+}(\text{div}\beta) Vu^{2} (r^{2}-|x|^{2})^{\alpha+1}\,dx
				+
				\frac{1}{r}\int_{B_r^+}\langle\beta, D( (r^{2}-|x|^{2})^{\alpha+1})\rangle Vu^{2}\,dx\nonumber
				\\&\quad+
				\frac{1}{r}\int_{B_r^+}(\beta \cdot DV)u^{2} (r^{2}-|x|^{2})^{\alpha+1}\,dx\label{eqn-210620-0939-1}
				\\&
				=I_{131}+I_{132}+I_{133}\label{eqn-210620-0939-2}
				,
		\end{align}
		where in \eqref{eqn-210620-0939-1}, we used the fact $$\beta\cdot n=(Ax\cdot n)/|x|^2=0\,\,\text{on}\,\,\Gamma_r.$$
		By \eqref{eqn-210616-1252-2}, we get
		\begin{equation}\label{eqn-210620-0949-1}
		\begin{split}
		I_{131}
		&=
		\frac{1}{r}\int_{B_r^+}\left(d+\epsi(r)O(1)\right) Vu^{2} (r^{2}-|x|^{2})^{\alpha+1}\,dx
		=
		\frac{d}{r}I_2 + \frac{\epsi(r)O(1)}{r}\It_2.
		\end{split}
		\end{equation}												
		Noting \eqref{eqn-210620-0121-1}, we have
		\begin{equation*}
		\begin{split}
		I_{132}
		&=
		-\frac{2(\alpha+1)}{r}\int_{B_r^+}|x|^2(r^{2}-|x|^{2})^{\alpha} Vu^{2}\,dx
		=
		\frac{2(\alpha+1)}{r}\int_{B_r^+}(-r^2 + r^2-|x|^2)(r^{2}-|x|^{2})^{\alpha} Vu^{2}\,dx
		\\&=
		-2(\alpha+1)r\int_{B_r^+} Vu^{2}(r^{2}-|x|^{2})^{\alpha}\,dx 
		+
		\frac{2(\alpha+1)}{r}I_2.
		\end{split}
		\end{equation*}
		 For $I_{133}$, noting \eqref{eqn-210912-0648} and \eqref{eqn-210620-0121-2}, we have
			\begin{align}\label{e.w06191}
			|I_{133}|
			\leq
			\frac{1}{r}\int_{B_r^+}|\beta||DV|u^{2} (r^{2}-|x|^{2})^{\alpha+1}\,dx \leq \lambda^{-2}r^2MH
			.
			\end{align}
		Substituting \eqref{eqn-210620-0949-1}-\eqref{e.w06191} back to \eqref{eqn-210620-0939-2}, we have
		\begin{equation}\label{eqn-210627-1156-3}
				I_{13}
				\geq											
				\frac{d+2\alpha+2}{r}I_{2}
				-C\frac{\epsi(r)}{r}\It_2
				-
				 Cr^2MH
				-
				2(\alpha+1)r\int_{B_r^+} Vu^{2}(r^{2}-|x|^{2})^{\alpha}\,dx 
				.
		\end{equation}
		Hence, combining \eqref{eqn-210627-1156-1}, \eqref{eqn-210627-1156-2}, and \eqref{eqn-210627-1156-3} yields
		\begin{equation}\label{eqn-210620-1029-3}
		\begin{split}
		I_1' 
		&= 
		\frac{2(\alpha+1)}{r}I_1 +I_{11} + I_{12} +I_{13}
		\\&\geq
		\frac{d+2\alpha}{r}I -C \frac{\epsi(r)}{r}(I_1 + \It_2 + \It_3)
		+
		\frac{2}{r}I_2 + \frac{1}{r}I_3
		-
		C\big(I_1 + \alpha H + M_\eta^2H\big)
		-
		 Cr^2MH
		\\&\quad+
		\frac{4(\alpha+1)}{r}\int_{B_r^+}(ADu\cdot x)^2\mu^{-1} (r^{2}-|x|^{2})^{\alpha}\,dx
		\\&\quad-
		2(\alpha+1)r\int_{B_r^+} Vu^{2}(r^{2}-|x|^{2})^{\alpha}\,dx 
		+
		2(\alpha+1)r\int_{\Gamma_r}u^2\eta(r^2-|x|^2)^{\alpha}
		.
		\end{split}
		\end{equation}
		By Lemma \ref{lem-210620-1120}, we obtain
		\begin{equation}\label{eqn-210703-0550}
		\frac{2}{r}I_2 + \frac{1}{r}I_3 \geq - \frac{2}{r}\It_2 - \frac{1}{r}\It_3 \geq -CMrH - \frac{1}{2}I_1 - \frac{\alpha}{2} H - CM_\eta^2H.
		\end{equation}
		Combining \eqref{eqn-210620-1029-1}, \eqref{eqn-210620-1029-2}, \eqref{eqn-210620-1029-3}, and \eqref{eqn-210703-0550}, we have
		\begin{equation*}
		\begin{split}
		I'
		&\geq
		\frac{d+2\alpha}{r}I
		-
		C\big(1+ \frac{\epsi(r)}{r}\big)(I_1 + \It_2 + \It_3)
		-
		C\big(\alpha  + rM + M_\eta^2\big)H
		\\&\quad+
		\frac{4(\alpha+1)}{r}\int_{B_r^+}(ADu\cdot x)^2\mu^{-1} (r^{2}-|x|^{2})^{\alpha}\,dx
		\end{split}
		\end{equation*}
		Applying Lemmas \ref{lem-210620-1120} and \ref{lem-210621-1219}, we obtain
		\begin{equation*}
			I' \geq \frac{d+2\alpha}{r}I
			-
			C\frac{\tepsi(r)}{r}I
			-
			C\frac{\tepsi(r)}{r}(Mr + \alpha + M_\eta^2)H
			+
			\frac{4(\alpha+1)}{r}\int_{B_r^+}(ADu\cdot x)^2\mu^{-1} (r^{2}-|x|^{2})^{\alpha}\,dx.
		\end{equation*}
		Here,
		\begin{equation*}
			\tepsi(r)=r+\epsi(r).
		\end{equation*}
	\end{proof}
	
	\subsection{Proof of Proposition \ref{prop-210127-1144} and Lemma \ref{lem-211211-1209}}\label{sec-210714-1249}
	Now we are ready to give the proof of Proposition \ref{prop-210127-1144}. 
	\begin{proof}[Proof of Proposition \ref{prop-210127-1144}]
	Combining Lemmas \ref{lem-210615-1157} and Lemma \ref{lem-210628-1214}, we have
		\begin{align*}
			N'H^2 
			&=
			I'H - IH' 
			\\&\geq
			H\big(\frac{d+2\alpha}{r}I
			-
			C\frac{\tepsi(r)}{r}I
			-
			C\frac{\tepsi(r)}{r}(Mr + \alpha + M_\eta^2)H
			\\&+
			\frac{4(\alpha+1)}{r}\int_{B_r^+}(ADu\cdot x)^2\mu^{-1} (r^{2}-|x|^{2})^{\alpha}\,dx\big)
			-
			I\big(\frac{2\alpha + d + C\epsi(r)}{r}H(r) + \frac{I(r)}{(\alpha+1)r}\big)
			\\&\geq
			-
			C\frac{\tepsi(r)}{r}IH 
			- 
			\frac{I^2}{(\alpha+1)r} 
			+ 
			\frac{4(\alpha+1)}{r}\big(\int_{B_r^+}(ADu\cdot x)^2\mu^{-1} (r^{2}-|x|^{2})^{\alpha}\,dx\big)H
			\\&\quad-
			C\frac{\tepsi(r)}{r}(Mr + \alpha + M_\eta^2)H^2
			.
		\end{align*}
		From \eqref{eqn-210606-1036} and the Cauchy-Schwarz inequality,
		\begin{equation*}
			\begin{split}
				\frac{I^2(r)}{(\alpha+1)r} &= \frac{1}{r(\alpha+1)}\big(2(\alpha+1)\int_{B_r^+} \left<ADu, x\right> u(r^{2}-|x|^{2})^{\alpha}\big)^2\\
				&\leq \frac{4(\alpha+1)}{r}\big(\int_{B_r^+}|\left<ADu, x\right>|^2\mu^{-1}(r^2-|x|^2)^\alpha\big) \big(\int_{B_r^+}u^2(r^2-|x^2|)^\alpha\mu\big)\\
				&= \frac{4(\alpha+1)}{r}\big(\int_{B_r^+}|\left<ADu, x\right>|^2\mu^{-1}(r^2-|x|^2)^\alpha\big)H.
			\end{split}
		\end{equation*}
Hence,
		\begin{equation*}
			N' \geq -
			C\frac{\tepsi(r)}{r}N -
			C\frac{\tepsi(r)}{r}(Mr + \alpha + M_\eta^2).
		\end{equation*}
		The proposition is proved.
	\end{proof}
\begin{proof}[Proof of Lemma \ref{lem-211211-1209}]
	We multiply both sides of \eqref{eqn-210628-0104} by
	\begin{equation*}
		e^{CI_\epsi(r)} \quad\text{with}\quad I_\epsi(r):=\int_0^r\frac{\tepsi(t)}{t}\,dt,
	\end{equation*}
Clearly, $I_\epsi(r)$ is increasing in $r$. From \eqref{eqn-210606-1003}, $I_\epsi(1)\leq C$. From these, for any $r\in(0,1]$,
\begin{equation*}
	\begin{split}
	(N(r)e^{CI_\epsi(r)})' 
	&\geq
	-e^{CI_\epsi(r)}C\frac{\tepsi(r)}{r}(Mr + \alpha + M_\eta^2)
	\\&\geq
	-C\frac{\tepsi(r)}{r}(M+\alpha+M_\eta^2).
	\end{split}	
\end{equation*}
Hence,
\begin{equation}\label{eqn-211211-1248-1}
	\frac{d}{dr}\big(N(r)e^{CI_\epsi(r)} + CI_\epsi(r)(M+\alpha+M_\eta^2)\big) \geq 0.
\end{equation}
Furthermore, from Lemma \ref{lem-210615-1157} and $N = I/H$, we have
\begin{equation}\label{eqn-211211-1248-2}
	N(r) = (\alpha+1)r\frac{d}{dr}\log(H(r)) - (\alpha+1)(2\alpha+d) + O(1)(\alpha+1)\epsi(r).
\end{equation}
Combining \eqref{eqn-211211-1248-1}-\eqref{eqn-211211-1248-2} and noting \eqref{eqn-210606-1003}, we have for any $\tau<\rho$,
\begin{equation}\label{eqn-211224-0319}
	\begin{split}
	&(\alpha+1)\big(r e^{CI_\epsi(r)} \frac{d\log(H(r))}{dr}\big)\Big\vert_{r=\tau} - \widetilde{C}(\alpha+1)(2\alpha+d) - \widetilde{C}(\alpha+1)\epsi(\tau)
\\&\leq	N(\tau)e^{CI_\epsi(\tau)} + CI_\epsi(\tau)(M+\alpha+M_\eta^2) 
\\&\leq N(\rho)e^{CI_\epsi(\rho)} + CI_\epsi(\rho)(M+\alpha+M_\eta^2)
\\&\leq (\alpha+1)\big(r e^{CI_\epsi(r)} \frac{d\log(H(r))}{dr}\big)\Big\vert_{r=\rho} + \widetilde{C}(\alpha+1)\epsi(\rho) + \widetilde{C}(M+\alpha+M_\eta^2)
	\end{split}
\end{equation}
Here $C$ is the constant coming from \eqref{eqn-211211-1248-1} and $\widetilde{C}=\widetilde{C}(d,\lambda,I_\epsi)>0$.
Now we construct the new variable
\begin{equation*}
	s(r) = -\int_r^1t^{-1}e^{-CI_\epsi(t)}\,dt.
\end{equation*}
Due to the boundedness, positivity, and monotonicity of $I_\epsi(r)$, we have
	\begin{equation*}
		\frac{ds}{d\log(r)} = r\frac{ds}{dr} = e^{-CI_\epsi(r)} \in [e^{-CI_\epsi(1)},1)\quad\text{and is decreasing in}\,\,r.
\end{equation*}
Furthermore, $\epsi(r)\,ds \approx \epsi(r)r^{-1}\,dr$ is integrable according to \eqref{eqn-210606-1003}.
From these we can integrate both sides of \eqref{eqn-211224-0319} against the new variable $s$ to obtain: for any $r,R\in(0,1)$ and $\kappa>1$ satisfying
\begin{equation*}
	0<r<\kappa r < R/\kappa < R<1,
\end{equation*}
there exists some constant $\delta\in (e^{-CI_\epsi(1)},1)$, depending on $(d,\lambda, I_\epsi, \kappa, r, R)$, such that
\begin{equation*}
	(\alpha+1)\log\left(\frac{H(\kappa r)}{H(r)}\right) \leq C(\alpha+1)\log\left(\frac{H(R)}{H(R/\kappa^{\delta})}\right) + C(M+\alpha+\alpha^2+M_\eta^2).
\end{equation*}
The lemma is proved.
\end{proof}

	\section{Proof of Theorem \ref{thm-210208-0351}(a)}\label{sec-210714-0108}
In this section, we will construct a change of variable which takes the original problem on the curved domain to the one on the half space. Hence, Lemma \ref{lem-211211-1209} will lead to the doubling property of the original problem, which further proves Theorem \ref{thm-210208-0351}(A).
	\subsection{Flattening the boundary}\label{sec-flat}
	In this section, we construct a change of variable to transform \eqref{eqn-21-0109-1122} on $\Omega\cap B_2$ to \eqref{eqn-21-0706-0210} on $\bR^d\cap B_{R_0}$ for some $R_0>0$, with coefficients satisfying \eqref{e.w11301}. Our construction is adapted from \cite{AE}.
	
	Since $\p\Omega\in C^{1,1}$, we know that there exists a global $C^{1,1}$ distance function $d(x)$. However, in application it is more convenient to work with distance with higher derivatives in the interior, which is the so-called ``regularized distance''.
	\begin{lemma}\label{lem-211029-0853}
		There exists a distance function $\rho \in C^{1,1}(\overline \Omega)$ such that
		\begin{enumerate}
			\item $cd(x)\leq \rho(x)\leq Cd(x)$.
			\item $|D\rho|\geq C_0$ on $\p\Omega$.
			\item $|D^{k+2} \rho|\rho^{k} \in L^{\infty}(\Omega)$, for any $k\geq 1$.
		\end{enumerate}
	\end{lemma}
Such $\rho$ can be constructed by simply mollifying $d(x)$ and thus we omit the proof. Now we construct our change of variable. Without loss of generality, suppose locally $D_{x_d}\rho\neq 0$.

\textbf{Step (i)}: Flattening the boundary. We set
\begin{equation*}
\begin{cases}
z'=x',
\\z_{d}
=
\rho(x)
,
\end{cases}
\end{equation*}	
which locally takes $\Omega$ to $\{z_d>0\}$. Clearly, such transformation is two-way $C^\infty(\Omega)\cap C^{1,1}(\overline{\Omega})$, with norms depending only on $\Omega$.  Direct computation shows
\begin{equation*}
\det \left(\frac{\partial z}{\partial x}\right)
=
\det \begin{pmatrix}
I_{d-1} & 0\\
(D_{x'}\rho)^{T}& D_{x_d}\rho
\end{pmatrix}
=
D_{x_d}\rho
\end{equation*}	
and
\begin{equation*}
d\sigma_{x}
=
\frac{|D_x\rho|}{D_x\rho\cdot \vec{e}_{d}}d\sigma_{z}
=
\frac{|D_x\rho|}{D_{x_d}\rho}d\sigma_{z}\quad\text{on}\,\,\p\Omega=\{\rho=0\}=\{z_d=0\}
.
\end{equation*}	
In $z$-coordinates, the problem becomes
\begin{equation*}
\begin{cases}
\text{div}_{z} \left( A_{(z)}D_{z}u\right)=\frac{1}{D_{x_d}\rho}Vu\quad\text{in}\,\,B_r^{+}(0),
\\
\big((\vec{n}_{A_{(z)}}\cdot D_{z})u \big) \cdot (-\vec{e}_d)
=
\frac{|D_x\rho|}{D_{x_d}\rho} \eta  u\quad \text{on}\,\,B_r(0)\cap\{z_d=0\}
.
\end{cases}
\end{equation*}
where
\begin{equation*}
	A_{(z)}= \frac{1}{D_{x_d}\rho}\frac{\partial z}{\partial x}A(\frac{\partial z}{\partial x})^T  
\end{equation*}
and its conormal vector
\begin{equation*}
	\vec{n}_{A_{(z)}} = \frac{1}{D_{x_d}\rho}\frac{\partial z}{\partial x}A(\frac{\partial z}{\partial x})^T(-\vec{e}_d).
\end{equation*}

\textbf{Step (ii)}: Mapping the conormal vector field to normal directions. More precisely, we aim to construct new coordinates $y=(y',y_d)$, such that we have the push-forward
\begin{equation}\label{eqn-210903-0648}
	\big(\frac{\p z}{\p y}\big)_\# \left(\frac{D_{x_d}\rho}{|D_x\rho|}\frac{\p}{\p y_d}\right) 
	= 
	\vec{n}_{A_{(z)}}\cdot 
	\frac{\p}{\p z}\quad\text{on}\,\,\{z_d=0\}\big(=\{y_d=0\}\big).
\end{equation} 
We will achieve this in two steps. First, we construct $w=(w',w_d)$ with
\begin{equation*}
\left(\frac{\p z}{\p w}\right)_\# \left(\frac{\p}{\p w_d}\right) 
= 
\frac{1}{\mu_d}\vec{n}_{A_{(z)}}\cdot \frac{\p}{\p z}\quad\text{on}\,\,\{z_d=0\}\big(=\{y_d=0\}\big),
\end{equation*} 
where by ellipticity
\begin{equation*}
	\mu_d := \left<\vec{n}_{A_{(z)}},-\vec{e}_d\right>=\left<\frac{\partial z}{\partial x}A(\frac{\partial z}{\partial x})^T (-\vec{e}_d),(-\vec{e}_d)\right> \geq C\lambda.
\end{equation*}
Thus, we can define $\vec{\tau} =\vec{\tau}(z',0)\in \bR^{(d-1)\times 1}$ on $\{z_d=0\}$ to be the first $(d-1)$ components of $\vec{n}_{A_{(z)}}/\mu_d$, i.e.,
\begin{equation*}
\begin{pmatrix}
\vec{\tau}(z',0)\\
1
\end{pmatrix} = \frac{\vec{n}_{A_{(z)}}}{\mu_d}\bigg|_{z_d=0}.
\end{equation*}
	From our construction, 
	\begin{equation*}
		\big(\frac{\p z}{\p x}\big)^T\vec{e}_d = D\rho.
	\end{equation*}
	Noting in $x$-coordinates the unit outward normal $\vec{n} = -D\rho/|D\rho|$ and \eqref{eqn-210903-0627}, we have
	\begin{equation*}
		\vec{n}_{A_{(z)}} = \frac{1}{D_{x_d}\rho}\big(\frac{\p z}{\p x}\big)A\big(\frac{\p z}{\p x}\big)^T(-\vec{e}_d) \in C^{0,1}.
	\end{equation*}
	Hence, $\vec{\tau}\in C^{0,1}$. Now, we can extend $\vec{\tau}$ and $\mu_d$ to the upper half space, then mollify to obtain $\widetilde{\vec{\tau}}, \widetilde{\mu}_d\in C^{0,1}(\overline{\bR^d_+})\cap C^\infty(\bR^d_+)$, with
	\begin{equation*}
	z_d^{k-1}|D^k\widetilde{\vec{\tau}}|\leq C(\|\tau\|_{C^{0,1}},\Omega)\,\,\text{and}\,\, z_d^{k-1}|D^k\widetilde{\mu}_d|\leq C(\|\mu\|_{C^{0,1}}, \Omega)\quad \forall k\in\mathbb{N}_+.
	\end{equation*}
Define
		\begin{equation*}
			\begin{cases}
			w'
			=
			z'-z_{d} \widetilde{\vec{\tau}},
			\\
				w_{d}=z_{d}
				.
			\end{cases}
		\end{equation*}
		Clearly such transformation takes $\bR^d_+$ to $\bR^d_+$ with 
		\begin{equation*}
			\frac{\partial w}{\partial z}
			=
			\begin{pmatrix}
				I_{d-1} -z_dD_{z'}\widetilde{\vec{\tau}} & -\widetilde{\vec{\tau}} - z_dD_{z_d}\widetilde{\vec{\tau}}\\
				0& 1
			\end{pmatrix}.
		\end{equation*}
		Restricted on the boundary, i.e. on $\{w_d=0\}=\{z_d=0\}$,
		\begin{equation*}
		\frac{\p z}{\p w} 
		= \left(\frac{\p w}{\p z}\right)^{-1} 
		= \begin{pmatrix}
		I_{d-1}  & \widetilde{\vec{\tau}}\\
		0& 1
		\end{pmatrix}.
		\end{equation*}
        Therefore in $w$-coordinates, the boundary condition becomes
        \begin{equation*}
        	\frac{\p u}{\p w_d} = \frac{1}{\mu_d}\frac{|D_x\rho|}{D_{x_d}\rho}\eta u.
        \end{equation*}
        One could check that the transformation $z\mapsto w$ is also two-way $C^{1,1}$. For example, the most singular term in $\frac{\p^2 w}{\p z^2}$ is $-z_d D^2_{w}\widetilde{\tau}$, which is bounded according to Lemma \ref{lem-211029-0853}.
        
{\textbf{Step (iii)}}: We are left to do one more normalization by letting
        \begin{equation*}
        \begin{cases}
        y'=w',
        \\y_d
        =
        \frac{\widetilde{\mu}_d D_{x_d}\rho}{|D_x\rho|}w_d
        .
        \end{cases}
        \end{equation*}
        As before, using Lemma \ref{lem-211029-0853} and $w_d\approx z_d$, one may check that $w\mapsto y$ is two-way $C^{1,1}$. Clearly, such $y$ satisfies \eqref{eqn-210903-0648}, as we desired.
        
		In the new coordinates $y$, the conditions at the beginning of Section \ref{sec-210706-0145} can be simply verified. It is worth mentioning that the change of variable here does not depend on any information of $V$ and $\eta$.
	
	\subsection{Doubling inequality and proof of Theorem \ref{thm-210208-0351} (A)}\label{sec-doubling}
	Suppose $u\in W^{1,2}(\Omega_2)$ is a solution to \eqref{eqn-21-0109-1122}. Let
	\begin{equation*}
			h(r)=\int_{\Omega_r}u^2.
	\end{equation*}
We first prove the doubling inequality.
\begin{proposition}\label{prop-211224-0522}
		Let $\Omega\in C^{1,1}$ and $0\in\p\Omega$. Suppose that \eqref{eqn-210706-0215}, \eqref{eqn-210912-0648}, and Assumption \ref{ass-210912-0655} (with $x_0=0$) hold. Then there exist constants $R_0, r_0\in(0,1)$ depending on $d,\lambda, I_\epsi, \|\p\Omega\|_{C^{1,1}}$, and $\|A\vec{n}\|_{C^{0,1}}$, such that for any $r\in(0,R_0)$,
	\begin{equation}\label{eqn-211206-0532}
		\begin{split}
			\frac{h(2r)}{h(r)} \leq \big(\frac{h(1)}{h(r_0)}\big)^C e^{C(\sqrt{M} + M_\eta+1)},
		\end{split}
	\end{equation}
where $C=C(d, \lambda, I_\epsi, \|\p\Omega\|_{C^{1,1}}, \|A\vec{n}\|_{C^{0,1}})$.
\end{proposition}
\begin{proof}
	Let $\Phi: \Omega\cap B_2\mapsto \bR^d_+$ be the change of variable constructed in Section \ref{sec-flat}. We can find a constant $c_0\in(0,1)$, such that
	\begin{equation}\label{eqn-210706-0356}
		B_{c_0r}^+ \subset \Phi(\Omega_r)\subset B_{c_0^{-1}r}^+\quad\forall r\in(0,2).
	\end{equation}
Hence, $u\circ\Phi^{-1}$ solves a Robin problem on $B_{2c_0}^+\subset\bR^d_+$, for which all the conditions in Section \ref{sec-210706-0145} can be simply verified.
We denote
	\begin{equation*}
		\tilde{h}(\rho)=\int_{B_\rho^+}|u\circ\Phi^{-1}|^2\quad \text{for}\,\,\rho\in(0, 2c_0).
	\end{equation*}
Furthermore, replacing $u$ with $u\circ\Phi^{-1}$, we recall $H, I,$ and $N$ as in \eqref{eqn-210606-1036}. Clearly,
\begin{equation*}
	H(\rho)\leq \lambda^{-1}\rho^{2\alpha}\tilde{h}(\rho),\quad \tilde{h}(\rho)\leq\frac{\lambda^{-1}H(\tau)}{(\tau^2-\rho^2)^\alpha}\quad\forall \rho<\tau.
\end{equation*}
From this and Lemma \ref{lem-211211-1209} with $R=c_0$, a large $\kappa$ to be chosen later, and $r<c_0/\kappa$, we have
\begin{equation*}
	\begin{split}
		\log\left(\frac{\tilde{h}(\kappa r/2)}{\tilde{h}(r)}\right) 
		&\leq 
		\log\left(\frac{H(\kappa r)}{H(r)}\right) - 2\log(\lambda) -C\alpha\log(\kappa)
		\\&\leq 
		C\log\left(\frac{H(R)}{H(R/\kappa^{\delta})}\right) + C(\alpha+1)^{-1}(M+\alpha+\alpha^2+M_\eta^2) 
		\\&\leq 
		C\log\left(\frac{\tilde{h}(R)}{\tilde{h}(R/(2\kappa^{\delta}))}\right) + C\alpha\log(\kappa) +  C(\alpha+1)^{-1}(M+\alpha+\alpha^2+M_\eta^2).
	\end{split}
\end{equation*}
Choosing $\alpha = 1 + \sqrt{M} + M_\eta$ and raising to the exponential, we obtain
\begin{equation}\label{eqn-22-0108-0856}
		\frac{\tilde{h}(\kappa r/2)}{\tilde{h}(r)}
		\leq 
		\left(\frac{\tilde{h}(c_0)}{\tilde{h}(c_0/(2\kappa^{\delta})}\right)^C \kappa^{1+\sqrt{M}+M_\eta}e^{C(1+\sqrt{M}+M_\eta)}.
\end{equation}
Now we transform back to $h$ on the curved domain, and fix $\kappa = 4c_0^{-2}$. Note that from \eqref{eqn-210706-0356}, for any $r<R_0:=c_0^2/4$, we have
\begin{equation*}
	B_{c_0r}^+ \subset \Phi(\Omega_r) \subset \Phi(\Omega_{2r}) \subset B_{\kappa c_0r/2}^+
\end{equation*}
and for $r_0 := c_0^4/8$,
\begin{equation*}
	\Phi(\Omega_{r_0}) \subset B_{c_0/(2\kappa)}^+ \subset B_{c_0/(2\kappa^{\delta})}^+ \subset  B^+_{c_0} \subset \Phi(\Omega_1).
\end{equation*}
Replacing $r$ with $c_0 r$ in \eqref{eqn-22-0108-0856}, using the above two inclusion relations, we have 
\begin{equation*}
	\frac{h(2r)}{h(r)} 
	\leq 
	C\frac{\tilde{h}(\kappa c_0 r/2)}{\tilde{h}(c_0r)} 
	\leq 
	\left(\frac{\tilde{h}(c_0)}{\tilde{h}(c_0/(2\kappa^{\delta}))}\right)^C e^{C(1+M+\sqrt{M_\eta})}
	\leq
	\left(\frac{h(1)}{h(r_0)}\right)^C e^{C(1+M+\sqrt{M_\eta})}.
\end{equation*}
The proposition is proved.
\end{proof}

	Now we are ready to give the proof of Theorem \ref{thm-210208-0351}(A). 
	\begin{proof}[Proof of Theorem \ref{thm-210208-0351}(A)]
We iterate the doubling inequality in Proposition \ref{prop-211224-0522}: Suppose $r\in [2^{-k-1}R_0, 2^{-k}R_0)$, then
		\begin{equation*}
			\int_{\Omega_{2^{-(k+1)r}}}u^2 \leq (e^{C(\sqrt{M} + M_\eta + \log(h(1)/h(r_0)) +1)})^{k+1} \int_{\Omega_r}u^2.
		\end{equation*}
		Hence,
		\begin{equation*}
			\int_{\Omega_r}u^2 \geq e^{-(k+1)C(\sqrt{M} + M_\eta + \log(h(1)/h(r_0)) +1)}\int_{\Omega_{R_0}}u^2 \geq Cr^{C(\sqrt{M} + M_\eta + \log(h(1)/h(r_0)) +1)}\int_{\Omega_{R_0}}u^2.
		\end{equation*}
	The theorem is proved.
	\end{proof}

  \section{The case when $\eta=-\eta_0<0$}\label{sec-thm1b}
  In this section, we prove Theorem \ref{thm-210208-0351} (B).  The proof is very similar to that of Theorem \ref{thm-210208-0351} (A). The major difference lies in Lemma \ref{lem-210628-1214} for the half space case, where we eliminate the dependence on $M_\eta$ by proving the following lemma. Suppose that $u\in W^{1,2}(B_2^+)$ solves \eqref{eqn-21-0706-0210} with $\eta=-\eta_0<0$ and $H, I$ are defined as in \eqref{eqn-210606-1036}.
  \begin{lemma}\label{lem-210713-1058}
  	For any $r\in(0,1]$, we have
  	\begin{equation*}
  	I' \geq \frac{d+2\alpha}{r}I
  	-
  	C\frac{\tepsi(r)}{r}I
  	-
  	C\frac{\tepsi(r)}{r}rMH
  	+
  	\frac{4(\alpha+1)}{r}\int_{B_r^+}(\beta\cdot Du)(ADu\cdot x) (r^{2}-|x|^{2})^{\alpha}\,dx.
  	\end{equation*}
  \end{lemma}
\begin{proof}[Proof of Lemma \ref{lem-210713-1058}]
	Since the proof is almost the same with that of Lemma \ref{lem-210628-1214}, here let us only give a sketch while pointing out the modifications. Recall the decomposition $I=I_1+I_2+I_3$ in \eqref{eqn-210606-1039}. Noting that in this case,
	\begin{align*}
		I_3 =  \int_{\Gamma_r}\eta_{0} |u|^2(r^{2}-|x|^{2})^{\alpha+1}
	\end{align*}
	has a positive integrand, we have $\It_3=I_3$ for the majorants in \eqref{eqn-210713-1102}.

The computations of $I_2'$ and $I_3'$ in \eqref{eqn-210620-1029-1} and \eqref{eqn-210620-1029-2} stay the same. For $I_1'$, the same computation till \eqref{eqn-210714-1201} yields
	\begin{equation}\label{eqn-210714-1241}
		I_1' = \frac{2(\alpha+1)}{r}I_{1} + I_{11}+I_{12}+I_{13},
	\end{equation}
	where,
	\begin{equation*}
		\begin{split}
		I_{11} &= \frac{2}{r}\int_{B_r^+}\text{div}(\langle\beta, Du\rangle ADu ) (r^{2}-|x|^{2})^{\alpha+1}\,dx
		\\&= I_{111} + I_{112} + I_{113} + \frac{4(\alpha+1)}{r}\int_{B_r^+}(ADu\cdot x)^2\mu^{-1} (r^{2}-|x|^{2})^{\alpha}\,dx,
		\\I_{12} &= \frac{(d-2+\varepsilon(r)O(1))}{r}\int_{B_r^+} \left<ADu, Du\right> (r^{2}-|x|^{2})^{\alpha+1}\,dx ,
		\\I_{13} &= -
		\frac{2}{r}\int_{B_r^+}(\beta\cdot Du)Vu (r^{2}-|x|^{2})^{\alpha+1}\,dx.
		\end{split}
	\end{equation*}
	Since $\eta$ is not involved in the computations of $I_{12}$ and $I_{13}$, we still have \eqref{eqn-210627-1156-2} and \eqref{eqn-210627-1156-3}.
	Some modification is needed in computing $I_{11}$. Recall in \eqref{eqn-219713-1159} and \eqref{eqn-210713-1115},
	\begin{equation*}
		I_{11} = I_{111} + I_{112} + I_{113} + \frac{4(\alpha+1)}{r}\int_{B_r^+}(ADu\cdot x)^2\mu^{-1} (r^{2}-|x|^{2})^{\alpha}\,dx.
	\end{equation*}
	In this case
	\begin{equation*}
		I_{111} = -\frac{1}{r}\int_{\Gamma_r} u^2(\beta\cdot D_{T}\eta_0)(r^2-|x|^2)^{\alpha+1} = 0.
	\end{equation*}
	From \eqref{e.w11162} we see
	\begin{align}
		I_{112} &= - \frac{2(\alpha+1)}{r}\int_{\Gamma_r}u^2\eta_{0}(r^2-|x|^2)^{\alpha}(\beta\cdot x)\nonumber
		\\&= \frac{2(\alpha+1)}{r}I_3 - 2(\alpha+1)r\int_{\Gamma_r}u^2\eta_{0}(r^2-|x|^2)^{\alpha}.\nonumber
	\end{align}
	From \eqref{eqn-210714-1207-2} and the fact $\It_3=I_3$, we get
	\begin{align}
	I_{113}
	=
	\frac{d-1}{r}I_3+\frac{\epsi(r)O(1)}{r}I_3.\label{eqn-210713-1150}
	\end{align}
	Combining \eqref{eqn-210714-1241} - \eqref{eqn-210713-1150}, in contrast with \eqref{eqn-210620-1029-3}, now we have
	\begin{align}
	I_1'
	&\geq
	\frac{d+2\alpha}{r}I -C \frac{\epsi(r)}{r}(I_1 + \It_2 + I_3)
	+
	\frac{2}{r}I_2 + \frac{1}{r}I_3
	-
	rMH\nonumber
	\\&\quad+
	\frac{4(\alpha+1)}{r}\int_{B_r^+}(ADu\cdot x)^2\mu^{-1} (r^{2}-|x|^{2})^{\alpha}\,dx\nonumber
	\\&\quad-
	2(\alpha+1)r\int_{B_r^+} Vu^{2}(r^{2}-|x|^{2})^{\alpha}\,dx 
	-
	2(\alpha+1)r\int_{\Gamma_r}u^2\eta_{0}(r^2-|x|^2)^{\alpha}
	,\label{eqn-210714-1228}
	\end{align}
	By \eqref{eqn-210620-1114-1}, $I_1+I_3= I-I_2$, and the non-negativity of $I_1$ and $I_{3}$, we have
	\begin{equation}\label{eqn-210714-1240-1}
		0\leq I_1 + \It_2 + I_3 = I - I_2 + \It_2 \leq I + 2\It_2 \leq I + 2M\lambda^{-1} r^2H
	\end{equation}
	and
	\begin{equation}\label{eqn-210714-1240-2}
	\frac{2}{r}I_2 + \frac{1}{r}I_3 \geq -\frac{2}{r}\It_2 \geq -2M\lambda^{-1}rH .
	\end{equation}
	Substituting \eqref{eqn-210714-1240-1} and \eqref{eqn-210714-1240-2} back to \eqref{eqn-210714-1228}, and combining \eqref{eqn-210714-1241}, \eqref{eqn-210620-1029-1}, and \eqref{eqn-210620-1029-2}, we reach
	\begin{equation*}
	I' \geq \frac{d+2\alpha}{r}I
	-
	C\frac{\tepsi(r)}{r}I
	-
	C\frac{\tepsi(r)}{r}rMH
	+
	\frac{4(\alpha+1)}{r}\int_{B_r^+}(ADu\cdot x)^2\mu^{-1} (r^{2}-|x|^{2})^{\alpha}\,dx,
	\end{equation*}
	which finishes the proof of Lemma \ref{lem-210713-1058}.
\end{proof}
Also we recall the frequency function $N(r)$ defined in \eqref{eqn-210606-1036}. From Lemma \ref{lem-210713-1058}, we can prove the following proposition which is comparable to Proposition \ref{prop-210127-1144}, by following the proofs in Section \ref{sec-210706-0145}. Note that the dependence on $M_\eta$ is removed.
\begin{proposition}\label{prop-210713-1108}
	For any $r\in(0,1]$, we have
	\begin{equation*}
		N' \geq -
		C\frac{\tepsi(r)}{r}N -
		C\tepsi(r)M,
	\end{equation*}
where $C=C(d,\lambda)$.
\end{proposition}
Hence, Theorem \ref{thm-210208-0351}(b) can be proved by following the steps in Lemma \ref{lem-211211-1209} and Section \ref{sec-210714-0108}.

\section{Unique continuation from the boundary}\label{sec-thm2}
In this section, we prove Theorem \ref{thm-210912-0701}. Using the change of variable in Section \ref{sec-flat}, without loss of generality, we may assume $\Omega = \bR^d_+$ and $u$ solves \eqref{eqn-21-0706-0210}. We first prove a quantitative Cauchy uniqueness result.
\begin{lemma}\label{lem-210912-0802} 
	For any constant $\delta\in(0,1)$ and $r\in(0,1)$,
	\begin{equation}\label{eqn-210912-0711}
		\int_{B_r^+} |u|^2 \leq Cr\int_{\Gamma_{2r}} |u|^2 + \delta \int_{B_{2r}^+} |u|^2.
	\end{equation}
where $C=C(\lambda,d, \|a_{ij}\|_{W^{1,\infty}}, M, M_\eta, \delta)$.
\end{lemma}
Before its proof, let us explain that although our Robin problem has a sign-changing $\eta$, the following Caccioppoli inequality and the local maximum principle still hold:
\begin{equation*}
	\fint_{B_{1/2}^+} |Du|^2 \leq C\fint_{B_1^+} |u|^2,\quad \|u\|_{L_\infty(B_{1/2}^+)} \leq C\fint_{B_1^+}|u|^2.
\end{equation*}
The only difference is that the constant $C$ here also depends on $\|\eta\|_{L_\infty}$. For the proof of Caccioppoli inequality, we need to use the interpolation-type trace theorem in the form of Lemma \ref{lem-210620-1003} (without the weight) and a standard iteration argument. From the Caccioppoli inequality, the local maximum principle simply follows.
\begin{proof}[Proof of Lemma \ref{lem-210912-0802}]
	The proof is by contradiction, which is similar to that of \cite[Lemma~3.1]{AE}. By rescaling, we only need to prove for $r=1$. Now suppose \eqref{eqn-210912-0711} fails, then there exist a sequence of function $u^{(n)}\in W^{1,2}(B^+_2)$, satisfying
    \begin{equation}\label{eqn-210912-0735}
	\begin{cases}
		\operatorname{div}(A^{(n)}Du^{(n)}) =  V^{(n)}u^{(n)}  & \text{in }\, B_2^+,\\
		A^{(n)}Du^{(n)}\cdot \vec{n} = \eta^{(n)} u^{(n)}& \text{on }\, \Gamma_2,
	\end{cases}
	\end{equation}
where $A^{(n)} = (a_{ij}^{n})$ with their elliptic constants being bounded by $\lambda$ and Lipschitz constants being bounded by $M_{a}$, $\|V^{(n)}\|_{W^{1,\infty}}\leq M$, and $\|\eta^{(n)}\|_{W^{1,\infty}}\leq M_\eta$. After normalization, such $u^{(n)}$ satisfies
\begin{equation}\label{eqn-210912-0759}
	\int_{B_1^+} |u^{(n)}|^2 =1,
\end{equation}
\begin{equation}\label{eqn-210912-0739}
	\int_{\Gamma_{2}} |u^{(n)}|^2 \leq 1/n,\quad\text{and}\,\, \int_{B_{2}^+} |u^{(n)}|^2 \leq 1/\delta.
\end{equation}
From these, by the Caccioppoli inequality, 
\begin{equation}\label{eqn-210912-0757}
	\int_{B_{3/2}^+} |Du^{(n)}|^2 \leq C/\delta,
\end{equation}
where $C$ is independent of $n$. Denoting $v^{(n)}$ and $f^{(n)}$ to be the zero extensions of $u^{(n)}$ and $Du^{(n)}$ respectively. Clearly,
\begin{equation*}
	\|v^{(n)}\|_{L^2(B_{3/2})} + \|f^{(n)}\|_{L^2(B_{3/2})} \leq C/\delta.
\end{equation*}
From this, by passing to a subsequence, we have
\begin{equation}\label{eqn-210912-0743}
	(v^{(n)}, f^{(n)}) \rightarrow (v,f)\quad\text{weakly in}\,\,L^2\times (L^2)^d.
\end{equation}
Furthermore, from \eqref{eqn-210912-0739},
\begin{equation}\label{eqn-210912-0742-2}
	v^{(n)} \rightarrow 0\quad\text{strongly in}\,\,L^2(\Gamma_2).
\end{equation}
Also, we take extensions of $a_{ij}^{(n)}, V^{(n)}$, and $\eta^{(n)}$ to the lower half space, with their Lipschitz norms being bounded. By the Arzela-Ascoli theorem, passing to a subsequence, we have the following uniform convergence on $B_{3/2}$:
\begin{equation}\label{eqn-210912-0742}
	a_{ij}^{(n)}\rightarrow a_{ij},\quad V^{(n)}\rightarrow V,\quad\text{and}\,\,\eta^{(n)}\rightarrow \eta.
\end{equation}
Now we aim to prove
\begin{equation*}
	D_i(a_{ij}D_j v) = Vv\quad\text{in}\,\,B_{3/2}.
\end{equation*}
Since $u^{(n)}$ solves \eqref{eqn-210912-0735}, for any test function $\psi\in C^\infty_c(B_{3/2})$, we have
\begin{equation*}
	\int_{B_{3/2}^+} (a_{ij}^{(n)}D_ju^{(n)}D_i\psi + V^{(n)}u^{(n)}\psi) - \int_{\Gamma_{3/2}}\eta^{(n)}u^{(n)}\psi =0.
\end{equation*}
Hence,
\begin{equation*}
	\int_{B_{3/2}} (a_{ij}^{(n)}f^{(n)}_j D_i\psi + V^{(n)}v^{(n)}\psi) - \int_{\Gamma_{3/2}}\eta^{(n)}v^{(n)}\psi =0.
\end{equation*}
Passing $n\rightarrow 0$, noting \eqref{eqn-210912-0742}, \eqref{eqn-210912-0742-2}, and \eqref{eqn-210912-0743},
\begin{equation*}
	\int_{B_{3/2}} a_{ij} f_j D_i\psi + Vv\psi =0.
\end{equation*}
We are left to check $f=Dv$. Noting $Dv^{(n)} = f^{(n)} (= Du^{(n)})$ on $B_{3/2}^+$ and $Dv^{(n)} = f^{(n)} = 0$ on $B_{3/2}^-$, for any $\psi\in C_c^\infty(B_{3/2})$, we have
\begin{equation*}
	\begin{split}
		\left|\int_{B_{3/2}} v^{(n)}D_j\psi + f_j^{(n)}\psi\right| &= \left|\int_{\Gamma_{3/2}} v^{(n)}\vec{e}_j \cdot\vec{e}_d \psi\right|
		\\&\leq
		\|v^{(n)}\|_{L^2(\Gamma_{3/2})}\|\psi\|_{L^\infty} \leq \frac{1}{n}\|\psi\|_{L^\infty}\xrightarrow{\text{$n\rightarrow 0$}} 0.
	\end{split}
\end{equation*}
Here in the first line, we do integrate by parts on $B_{3/2}^+$. Passing $n\rightarrow 0$, we have $f=Dv$ on $B_{3/2}$. From above, we know that $v\in W^{1,2}(B_{3/2})$, and solves
\begin{equation*}
	\int_{B_{3/2}} (a_{ij} D_j v D_i\psi + Vv\psi) =0.
\end{equation*}
Since $v^{(n)}=0$ on $B_{3/2}^-$, we also have $v=0$ on $B_{3/2}^-$. From the standard weak unique continuation result, $v\equiv 0$.

Meanwhile, by \eqref{eqn-210912-0739}, \eqref{eqn-210912-0757}, the compact embedding $W^{1,2}\hookrightarrow L^2$, and \eqref{eqn-210912-0759},
\begin{equation*}
	1 = \int_{B_1^+} |v^{(n)}|^2 \rightarrow \int_{B_1^+} |v|^2.
\end{equation*}
This contradicts with $v\equiv 0$. Hence, the Lemma is proved.
\end{proof}
Now we are able to prove Theorem \ref{thm-210912-0701}.
\begin{proof}[Proof of Theorem \ref{thm-210912-0701}]
	As mentioned before, without loss of generality, we only need to consider the solution $u$ to \eqref{eqn-21-0706-0210} on the half space, vanishing at infinite order at the origin. From Lemma \ref{lem-210912-0802} and the doubling property in Proposition \ref{prop-211224-0522}, by choosing $\delta$ small enough, we have
	\begin{equation}\label{eqn-210912-0807-1}
		\int_{B_{r/2}^+} |u|^2 \leq Cr\int_{\Gamma_r} |u|^2.
	\end{equation}
From the local maximum principle and again the doubling property,
\begin{equation}\label{eqn-210912-0807-2}
	\left(\fint_{\Gamma_r} |u|^2\right)^{1/2} \leq \|u\|_{L^\infty(B_r^+)} \leq C \left(\fint_{B_{2r}^+} |u|^2\right)^2 \leq C\left(\fint_{B_{r/2}^+} |u|^2\right)^2.
\end{equation}
Combining \eqref{eqn-210912-0807-1} - \eqref{eqn-210912-0807-2}, we have the doubling inequality at the boundary, from which the unique continuation follows.
\end{proof}

\begin{proof}[Proof of Corollary \ref{cor-211001-0138}]
	We prove by contradiction. Suppose $u$ is a non-trivial solution to \eqref{eqn-21-0706-0210}, which vanishes on $\Sigma\subset\p\bR^d_+$ with the surface measure $\sigma(\Sigma)>0$. Without loss of generality, suppose $0$ is a density point of $\Sigma$, i.e.,
	\begin{equation}\label{eqn-211101-0335}
		\frac{|\Gamma_r\cap\Sigma^c|}{|\Gamma_r|}\rightarrow 0\quad\text{as}\,\,r\rightarrow 0.
	\end{equation}
We will prove that \eqref{eqn-210101-0254} holds with $x_0=0$. Hence, by Theorem \ref{thm-210912-0701} we have $u\equiv 0$, which leads to a contradiction.

For this, we apply H\"older's inequality, the local maximum principle as in \eqref{eqn-210912-0807-2}, and \eqref{eqn-210912-0807-1} to obtain
\begin{align}
	\left(\fint_{\Gamma_{r/2}}|u|^2\right)^{1/2} 
	&\leq C\|u\|_{L^\infty(\Gamma_{r/2})} \left(\frac{|\Gamma_{r/2}\cap\Sigma^c|}{|\Gamma_{r/2}|}\right)^{1/2} \nonumber
	\\&\leq C\left(\fint_{B_r^+}|u|^2\right)^{1/2} \left(\frac{|\Gamma_{r/2}\cap\Sigma^c|}{|\Gamma_{r/2}|}\right)^{1/2} \nonumber
	\\&\leq C\left(\fint_{\Gamma_{2r}}|u|^2\right)^{1/2} \left(\frac{|\Gamma_{r/2}\cap\Sigma^c|}{|\Gamma_{r/2}|}\right)^{1/2}.\label{eqn-211101-0338}
\end{align}
From \eqref{eqn-211101-0335}, for any $\epsi>0$, we can find $r_0=r_0(\epsi)$ small enough, such that 
\begin{equation*}
	\frac{|\Gamma_{r/2}\cap\Sigma^c|}{|\Gamma_{r/2}|}<\epsi\quad\forall r<r_0.
\end{equation*} 
Substituting this back to \eqref{eqn-211101-0338}, then iterating, we reach \eqref{eqn-210101-0254}. As explained before, this finishes the proof.
\end{proof}

	\section{Acknowledgments}
	Z. Li was partially supported by an AMS-Simons travel grant. W. Wang was partially supported by an AMS-Simons Travel Grant and NSF Grant DMS-1928930 while participating in the Mathematical Problems in Fluid Dynamics program hosted
	by the MSRI in Berkeley, California, during the Spring 2021 semester. W. Wang would like to thank the organizers of the Hamiltonian Methods in Dispersive and Wave Evolution Equations program at ICERM for kind hospitality during the Fall 2021 semester.


\end{document}